\documentclass[11pt, twoside]{article}
\usepackage{amsfonts,amssymb,amsmath,amsthm}
\usepackage{graphicx}
\usepackage[all]{xy}
\usepackage{pstricks,psfrag,xmpmulti,amscd, import}
 \divide\oddsidemargin by 2
\newtheorem{thm}{Theorem}[section]

\newtheorem{cor}[thm]{Corollary}
\newtheorem{lem}[thm]{Lemma}
\newtheorem{prop}[thm]{Proposition}

\theoremstyle{definition}
\newtheorem{defn}[thm]{Definition}

\newtheorem{rem}[thm]{Remark}
\newtheorem*{rem*}{Remark}

\newcommand{\C}{\mathbb{C}}
\newcommand{\N}{\mathbb{N}}

\newcommand{\D}{\mathbb{D}}

\newcommand{\ov}{\overline}
\newcommand{\ra}{\rightarrow}

\newcommand{\BB}{{\cal B}}
\newcommand{\CC}{{\cal C}}

\newcommand{\FF}{{\cal F}}

\newcommand{\NN}{{\cal N}}

\newcommand{\PP}{{\cal P}}

\renewcommand{\SS}{{\cal S}}
\newcommand{\TT}{{\cal T}}

\newcommand{\XX}{{\cal X}}

\renewcommand{\hat}{\widehat}

\newcommand{\s}{{\underline{s}}}

\newcommand{\BBt}{{\tilde{\BB}}}

\renewcommand{\hat}{\widehat}

\newcommand{\De}{\Delta}

\newcommand{\dDe}{\partial\Delta}

\newcommand{\Bhat}{\widehat{B}}
\newcommand{\BBhat}{\widehat{\BB}}

 \title{Singular values and non-repelling cycles for entire transcendental maps}
 
\vspace{5cm}
\author{Anna Miriam Benini \thanks{Partially supported by   the Marie Curie IEF grant H2020 703269 COTRADY.} {\small and}  N\'uria Fagella \thanks{Partially supported by the Spanish 
grant MTM2014-52209-C2-2-P and the Maria de Maeztu Excellence Grant MDM-2014-0445 of the BGSMath.} \\
\small Dept.~de Matem\`atica y Inform\`atica, Universitat de Barcelona\\ \small Barcelona Graduate School of Mathematics (BGSMath) \\ 
\small Gran Via 585 08007, Barcelona}

\begin{document}

\maketitle  
\begin{abstract}Let $f$ be a map with bounded set of singular values for which periodic dynamic rays exist and land. We prove that each non-repelling cycle is associated to a singular orbit  which cannot accumulate on any other non-repelling cycle. When $f$ has finitely many singular values this implies a refinement  of the Fatou-Shishikura inequality. Our approach is combinatorial in the spirit of the approach used by   \cite{Ki00}, \cite{BCLOS16} for polynomials. \let\thefootnote\relax

\end{abstract}

\section{Introduction}

Consider the iteration of an entire transcendental  map $f:\C\ra\C$. The map $f$ fails to be a covering due to the presence of \emph{singular values}, that is the set $S(f)$ of points near which not all inverse branches of $f^{-1}$ are well defined and univalent.  While the singular values of rational maps are always {\em critical values} (images of zeros of $f'$ or {\em critical points}), transcendental functions may have also {\em asymptotic values}, and we have that 

\[S(f)=\overline{\{\text{critical and asymptotic values for $f$}\}}.\]

 Recall that $s\in \C$ is an asymptotic value if there exists a curve $\gamma:[0,\infty)\ra\C$ such that $|\gamma(t)|\ra\infty$ as $t\ra\infty$ and $f(\gamma(t))\ra s$ as $t\ra\infty$ (for example, $s=0$ is an asymptotic value for the map $z\mapsto \exp(z)$, and the curve $\gamma$ can be taken to be the negative real axis).

Special classes of maps are singled out in terms of their set of singular values and will be important for our discusion. More precisely define 
\[
\mathcal{S} = \{ f:\C\to\C \text{\ entire} \mid \#S(f)<\infty\} 
\]
and
\[
\BB = \{ f:\C\to\C \text{\ entire} \mid S(f) \text{\ is bounded} \}, 
\]
Class $\SS$ and class $\BB$
 are known as the {\em Speiser} class and the {\em Eremenko Lyubich} class respectively. The elements in  $\mathcal{S}$ are called  {\em finite type} maps while those in $\mathcal{B}$  are  of {\em bounded type}.

Singular values play a crucial role to understand the dynamics of $f$. Indeed, using local dynamics it is possible to investigate the relation between  singular orbits and non-repelling cycles. For example it is well known that each cycle of parabolic and attracting basins contains a singular value \cite{Fa20,Mi}. Using normal families arguments one can also see that Cremer points and all points in the boundary of Siegel disks are contained in the accumulation set of the orbits of the singular values \cite{Fa20,Mi}.
 While the orbit of a singular value which belongs to an attracting or parabolic basin is fully contained  in the basin and only accumulates on the attracting or parabolic point associated to the basin,  it is not clear using only local theory that a unique singular orbit cannot accumulate, for example, on many Cremer cycles or cycles of boundaries of Siegel disks. 
 
Nevertheless, using perturbation arguments in the finite dimensional space of rational maps of degree $d\geq 2$,  Fatou \cite{Fa20} was able to show that the number of non-repelling cycles of a rational map of degree $d\geq 2$ is bounded by  $4d-2$, that is twice the number of its critical points counted with multiplicity. 
 Afterwards he conjectured that the optimal bound should be $2d-1$. 
With this goal in mind, Shishikura  \cite{Sh87} used quasiconformal surgery   to  perturb simultaneously all indifferent cycles to attracting ones, proving Fatou's conjecture, known nowadays as the \emph{Fatou-Shishikura inequality}. A simpler proof for polynomials using perturbation in the class of weakly polynomial-like maps can be found in \cite{DH85}, while  a  different approach using quadratic differentials is in \cite{Ep}.
Later on, it was proven in \cite[Theorem 5]{EL92} and \cite{GK} that  for an  entire transcendental map of finite type,  the number of non-repelling cycles is also bounded by the number of singular values of the map, proving the Fatou-Shishikura inequality for this class $\mathcal{S}$ of functions. 

All these results, like those on the non-existence of wandering domains for these finite-dimensional families,  are based on arguments in parameter space, and, despite giving a sharp bound, they do not provide dynamical information on how exactly the orbits of the singular values relate to the non -repelling cycles. For example, a priori  if the function has $q$ singular values and $q$ Cremer cycles there could be a unique singular value whose orbit  accumulates on the $q$ Cremer cycles while the remaining $q-1$ singular values do not accumulate on any Cremer point. 

A different  combinatorial approach  for polynomials with connected Julia sets was suggested by Kiwi \cite{Ki00}. Kiwi's approach is dynamical and associates to each non-repelling cycle a specific singular orbit. Observe that for polynomials the Julia set fails to be connected only when singular values are escaping. {This is a well understood case and  Kiwi's approach can be extended to the case in which  the Julia set is not connected (compare with \cite{BCLOS16})}. 
Very recently another dynamical approach involving  laminations and fibers was used in \cite{BCLOS16} to prove  a more general version of the Fatou-Shishikura inequality   for polynomials,  which takes into account also wandering branch continua. These new  approaches prove among others the following statement, which is slightly stronger than the classical Fatou-Shishikura inequality.  

\begin{prop}[\cite{BCLOS16}]\label{prop:Strong Shishi Poly} Let $P$ be a polynomial. Then any non-repelling cycle is associated to a weakly recurrent critical point, and distinct non-repelling cycles are associated to distinct  critical points.
\end{prop}
 
 It is a natural question whether these stronger versions for polynomials hold also for their transcendental analogues, that is for entire transcendental maps of finite type. And even further, whether or how far  they can be pushed when dealing with maps with an infinite number of singular values.  This approach is inlaid in the frame of understanding whether these relationships between singular values and non-repelling cycles are intrinsically of local nature or instead rely on the global structure of the families of  maps under consideration.
 
In this paper we combine the main results in \cite{BF15} with  classical normal families arguments and some combinatorics to give some answers to both of these questions.  Our main result (Theorem~\ref{Singular orbits trapped in basic regions}) applies to maps in class $\mathcal{B}$ with some additional conditions and, when applied to functions of finite type gives the transcendental version of Proposition~\ref{prop:Strong Shishi Poly}.

More precisely let $\BBhat \subset \BB$ be the class of entire transcendental functions defined in \cite{RRRS} for which the escaping set consists of curves, known as {\em dynamic rays} (see Section~\ref{sect:Background}). Class $\BBhat$ contains all functions in $\BB$ which are either of finite order or composition of functions of finite order in class $\BB$. A dynamic ray $G$ is \emph{periodic} if $f^n(G)=G$ for some $n\in\N$ and we say that a ray  \emph{lands} if $\ov{G}\setminus G=\{z_0\}\subset\C$. Finally a  \emph{rationally invisible} repelling point is a  repelling periodic point which is not the landing point of any periodic ray (see Remark~\ref{Invisible}).

\begin{prop}[Fatou-Shishikura inequality]\label{FS1}
Let $f$ be  an entire transcendental map in class $\BBhat$ with $N<\infty$ singular values, whose periodic rays land.  Then $f$ has at most $N$ non-repelling cycles.

Moreover for each non-repelling cycle $\mathcal{X}$ there exists at least one singular orbit  $\{f^n(s)\}_{n\in\N}$ with $s\in S(f)$  which is associated to $\mathcal{X}$ in the sense that 
\begin{enumerate}
\item[\rm (1)]  $\{f^n(s)\}_{n\in\N}$ accumulates on every element of  $\mathcal{X}$ (or, the boundaries of the  Siegel disks containing $\XX$ );
\item[\rm (2)] $\{f^n(s)\}_{n\in\N}$ does not accumulate on any other non-repelling cycle, nor on any rationally invisible repelling point, nor on any point on the boundary of a Siegel disk $\Delta\notin \XX$ (provided the point is not on a periodic ray or a periodic point).
\end{enumerate}
\end{prop}

Let us make some comments on the statement above.  For functions of finite type, Proposition~\ref{FS1} gives more information than the classical Fatou-Shishikura inequality in \cite{EL92}, since it associates individual singular orbits to individual non-repelling cycles rather than relying on a global counting of the number of singular values and the number of non-repelling cycles. However our proof requires the existence of dynamic rays (which is ensured by assuming that $f$ is a composition of functions of finite order in class $\BB$), and that periodic dynamic rays land. Although it is expected that all periodic rays land except  those whose forward orbit hits a singular value (as it is the case for polynomials), so far this has only been proven for the family $e^z+c$ \cite{Re06}. 
 The hypothesis of landing of periodic rays  is implied by the requirement that the {\em postsingular set}
 
\[\PP(f):=\ov{\bigcup_{n\in\N,s\in S(f)}f^n(s)}.\]

 is bounded, but the latter  is in general a much stronger hypothesis. For example, when the map is $f(z)=e^z+c$, having bounded postsingular set implies  non-recurrence of the asymptotic value.
 
Observe also that the statement is most meaningful when thinking of the interplay between singular orbits and  irrationally indifferent cycles. Indeed,  if $\XX$ is an attracting or parabolic cycle, the cycle of its attracting basins  contains a singular value \cite{Mi}, so to each attracting or parabolic cycle is associated trivially a singular orbit in the sense Proposition~\ref{FS1}. Conversely, it is obvious that the singular orbits accumulating on  a Cremer cycle or a cycle of boundaries of Siegel disks  cannot intersect any attracting or parabolic basin. 

Finally we remark that it is not known whether it is possible for boundaries of Siegel disks to contain periodic points or points belonging to  periodic  rays. If this were never the case, the special case at the end of (2) could be removed.

Since our methods are dynamical and do not rely on perturbations in finite-dimensional parameter spaces, we also obtain results for functions with infinitely many singular values. In order to be able to state such results we need some additional understanding of the structure of the dynamical plane for a function $f\in\BBhat$ whose periodic rays land. 

We say that a periodic dynamic ray $G$ (of period $p$)  {\em lands alone } if its landing point is not the landing point of any other dynamic ray (of period $p$). By recent results in \cite{BRG17}, the concept of landing alone is independent of the period, so we will omit it.   For any $p\geq 1$,  consider the closed  graph $\Gamma_p$ formed by rays which are fixed under  $f^p$ and which do not land alone, together with their landing points. The graph $\Gamma_p$ disconnects $\C$ into open unbounded regions, called \emph{basic regions} (for $f^p$).
 
The Separation Theorem in \cite{BF15} (see Section~\ref{sect:Background}, and \cite{GM} for polynomials), states that, even though the number of fixed rays by $f^p$ is infinite, the number of basic regions is finite, and each of them  contains exactly one of the following: either a parabolic basin   invariant under $f^p$; or an attracting  point fixed by $f^p$; or a Siegel   point fixed by $f^p$; or a Cremer   point fixed by $f^p$; or a repelling   point fixed by $f^p$ which is not the landing point of any fixed ray (of $f^p)$.    Following \cite{GM}, the attracting, Siegel, Cremer or repelling fixed point is called an \emph{interior fixed point} (for $f^p$),  and the invariant parabolic basin is called a \emph{virtual fixed point} (for $f^p$).

Our main theorem is the following:

\begin{thm}[Singular orbits trapped in basic regions]\label{Singular orbits trapped in basic regions}
 Let $f$ be an entire transcendental map in class $\BBhat$ whose periodic  rays land.   
Let $\mathcal{X}$ be a cycle of Siegel disks, attracting basins, parabolic basins or Cremer points of period $q$ and let $p$ be any multiple of $q$.  Let $\{B_i\}_{i=0\ldots q-1}$ be the basic regions for $f^p$ containing the elements of $\mathcal{X}$. Then, up to relabeling the indices, at least one of the following is true.
\begin{enumerate}
\item[\rm (1)] There exists a singular value $v$ for $f$ such that $v\in \bigcup_{i=0}^{q-1} B_i$, say $v\in B_0$, and such that $f^{n}(v)\in B_i$ whenever $n\mod q=i$. The orbit of $v$ accumulates either on the non-repelling cycle or on the boundary of the cycle of Siegel disks.

\item[\rm (2)]   There are infinitely many singular values $s_j$ for $f$ in at least one of the basic regions $B_i$, say $B_0$, and  a sequence $ n_j \underset{j\to\infty}{\longrightarrow}  \infty$ such that      $f^{n}(s_j)\in B_i$ whenever   $n\mod q=i$  for all $n\leq n_j $.  The orbits $\{f^{n}(s_j)\}_{j\in\N, n\leq n_j}$ accumulate either    on  the non-repelling interior cycle, or on the boundary of the associated Siegel disk. 
\end{enumerate}
The first case always occurs if $\mathcal{X}$ is attracting or parabolic or if $f$ has only  finitely many singular values. 

Moreover, in case (1), the orbit of $v$ does not accumulate on any other interior periodic cycle or  on any point on  the boundary of a Siegel disk $\Delta\notin \XX$ (provided the point is not on a periodic ray or a periodic point). 

\end{thm}
 
  Theorem~\ref{Singular orbits trapped in basic regions}  implies for example that if $f$ has infinitely many singular values, all but finitely many of which are in an attracting or parabolic basins, then $f$ has at most as many additional non-repelling cycles as the number of 'free' singular values.  
This fact is certainly not surprising but we believe it does not follow directly from the results by Eremenko and Lyubich \cite{EL92}, whose proof  uses perturbation in a finite-dimensional parameter space.  

If all singular values {but finitely many}  are escaping, the situation is not clear. For example,  let $\Delta$ be a bounded Siegel disk and let $\CC_n$ be a sequence of finite  coverings of $\partial\Delta$ by balls of radius $1/n\ra0$. Then if we have infinitely many escaping singular values $\{s_j\}_{j\in\N}$ we can make the singular orbit of $s_j$ visit all balls in $\CC_j$ before escaping infinity.

We believe that our methods work also when $f\in\BB\setminus\BBhat$. In that case the role of dynamic rays is taken by analogous, non-pathconnected objects called \emph{dreadlocks} \cite{BRG17}. One would need to prove the separation theorem using dreadlocks instead of rays and assume that periodic dreadlocks land. Such extension would remove the function theoretical assumption on $f$.
 
 The paper is organized as follows. Section 2 contains some of the background results that will be used throughout the paper. Section 3 is aimed at proving the main result (Theorem \ref{Singular orbits trapped in basic regions}) and some corollaries including Proposition \ref{FS1}. 

\subsection*{Acknowledgements} We wish to thank Lasse Rempe-Gillen and Mitsuhiro Shishikura for helpful discussions. We are grateful to the Institut de Matem\`atica de la Universitat de Barcelona (IMUB) for its hospitality.

\section{Background}\label{sect:Background}
Let $f$ be an entire transcendental function in class $\BB$ and let $D$ be a disk  containing $S(f)$.  The connected components of $f^{-1}(\C\setminus \ov{D})$ are called   \emph{tracts} and they are  unbounded and simply connected. By definition for any tract $T$ we have that $f:T\ra\C\setminus \ov{D}$ is an unbranched covering of infinite degree. Let $\TT$ denote the union of all tracts. 
One can easily show that there exists a (piecewise analytic) curve $\delta\subset \C\setminus(\ov{D\cup\TT})$ connecting $\delta$ to $\infty$ \cite{Rottenfusserthesis}. The preimages of $\C\setminus (\ov{D}\cup\delta)$ are called \emph{fundamental domains} and, since fibers are discrete, it follows that only finitely many tracts and fundamental domains intersect $D$.  For any fundamental domain $F$ we have that 
\[
f:F\ra \C\setminus (\ov{D}\cup\delta)
\]
 is a biholomorphism. We refer to \cite{EL92}, \cite{RRRS} and \cite{BF15} for details and other properties.  
 
Functions in class $\BB$ have expansive properties near infinity inside fundamental domains, as shown in the following lemma.  
 
\begin{lem}[{\cite[Proposition 2.6]{BF15}}]
\label{lem:cutting}Let $\FF=\{F_\alpha\}_{\alpha=1\ldots N}$ be a finite collection of fundamental domains. Then for any $R$ large enough  there exists an (analytic) Jordan curve $\gamma$ in $\{|z| > R\}$ such that    the preimages $\gamma_\alpha\subset F_\alpha$ of $\gamma$ are contained in the bounded connected component of $\C\setminus \gamma$.
 \end{lem}

Fundamental domains are endowed with a natural cyclic order, and can be used to define symbolic dynamics on the set of escaping points (or at least, those which stay far enough from  $D$ if the postsingular set is not bounded). This allows to define sets of escaping  points which share the same itinerary. If $f\in \BBhat$, it is shown in \cite{RRRS} that the tracts have a nice enough geometry to prove that these sets of escaping points are injective curves, called dynamic rays, and that each escaping point belongs to a dynamic ray or a preimage thereof. More precisely, let $\Sigma$ be the set of infinite sequences whose symbols are the fundamental domains of $f$,  and let $\sigma$ be the left-sided shift map acting on $\Sigma$. The elements of $\Sigma$ are called {\emph addresses}. An address is {\emph bounded} if it takes values over a finite family of fundamental domains, and \emph{periodic} if it is a periodic sequence.  The following statement summarizes the relation between these objects. 
%

\begin{thm}[\cite{RRRS}]Let $\hat\BB  \subset \BB$ be the class of maps formed by finite compositions of finite order maps in $\BB$.  Then there exists $\NN\subset\Sigma$ such that  the if $z\in I(f)$, then for $n$ large enough $f^n(z)$ belongs to an injective unbounded  curve $G_\s\subset I(f)$ called the \emph{dynamic ray} of address $\s$ for some $\s\in \NN$. The correspondence $\s\mapsto G_\s$ is injective  and $G_\s\cap G_{\tilde{s}}=\emptyset$ for $\s\neq\tilde{s}$. Dynamic rays  satisfy the relation  
\[
f(G_\s)=G_{\sigma\s}.
\]
\end{thm}

A dynamic ray $G_\s$ is \emph{periodic} if $\s$ is periodic, and has \emph{ bounded address} if $\s$ is bounded. The set $\NN$ of addresses  which are realized  depends on  the class of $f$ as defined in \cite{EL92,Re09}, but $\NN$ always contains the set of bounded addresses \cite{Re08,BK07} and addresses which are exponentially bounded in the sense of \cite{BRG17}. For the exponential family $\NN$ is completely characterized \cite{SZEsc}. For more on the characterization of $\NN$ see \cite{ABRG}.

The initial idea of finding curves in the escaping set goes back to \cite{DT86,DK,BK07}.
For functions with not as beautiful a  geometry as functions in class $\BBhat$, the role of dynamic rays is played by more general connected sets of escaping points called \emph{dreadlocks} \cite{BRG17}.
  
We say that an unbounded set $X$ is \emph{asymptotically contained}   in another unbounded set $U$ if and only if there exists $R$ such that $X\cap\{z\in\C: |z|>R\}\subset U$.
It is not hard to see that a ray $G_\s$ of address $\s=F_0F_1\ldots$ is asymptotically contained in the  fundamental domain $F_0$. Also, for each fundamental domain $F$ there exists a unique fixed ray with address $\ov{F}$ (see \cite{Re08}, \cite{BK07} and \cite[Lemma 2.3]{BF15}) and which is asymptotically contained in $F$. 
 
 The proof of Theorem~\ref{Singular orbits trapped in basic regions} uses strongly the following theorem.
 
\begin{thm}[Separation Theorem \cite{BF15}]\label{Separation Theorem Entire} Let  $p\geq 1$ and $f\in\BBhat$ and assume  that all fixed rays of $f^p$ land. Then there are finitely many basic regions for $f^p$, and each basic region  contains exactly one interior fixed point or virtual fixed point of $f^p$.
\end{thm}

Theorem~\ref{Separation Theorem Entire} was inspired by an analogous Separation Theorem by Goldberg and Milnor \cite{GM} for polynomials {\em with connected Julia sets}, a condition equivalent to requiring that the postcritical set is bounded and which implies that all periodic rays land. So in the transcendental setting, the assumption that all periodic rays land is weaker than than the hypothesis that the postsingular set is bounded. 

Goldberg-Milnor's Separation Theorem and Theorem~\ref{Separation Theorem Entire} have many corollaries, including that parabolic points are always landing points of periodic dynamic rays, and that  hidden components of a Siegel disk are preperiodic to the Siegel disk itself (see \cite{CR} and \cite{BF2} for an application of this fact to the existence of critical points on the boundary of Siegel disks). See also \cite{Ki00}.

\section{Singular Values and Basic Regions}\label{proof}

In this section we prove Theorem~\ref{Singular orbits trapped in basic regions} and  Proposition~\ref{FS1}.
Consider an entire transcendental map $f\in\BBt$ whose periodic rays land, and consider the basic regions for $f^p$ for some $p\in\N$. 
By the Separation Theorem  \ref{Separation Theorem Entire}, each basic region contains exactly one interior fixed point for $f^p$ or an attracting parabolic basin  fixed under $f^p$. 

We shall make a further distinction between different types of basic regions. See Figure \ref{fig:basicregions}.

\begin{defn}[Basic regions of transcendental and polynomial type]
 A basic region $B$ is called of \emph{polynomial type} if $B\cap (\C\setminus\TT)$ is bounded and of \emph{transcendental type} otherwise. 
\end{defn}

\begin{figure}[htb!]
\centering
\includegraphics[width=0.6\textwidth]{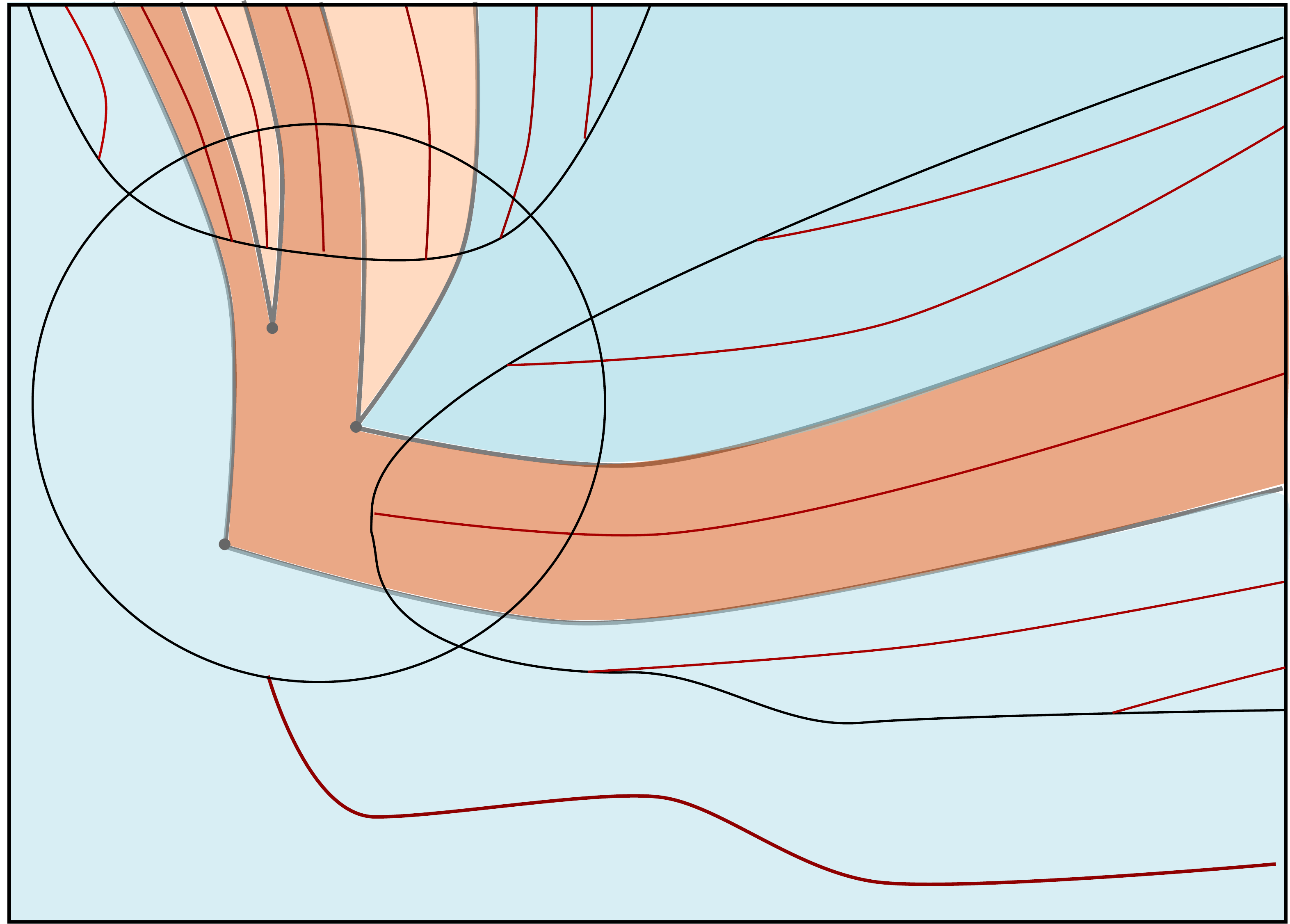}
\setlength{\unitlength}{0.6\textwidth}
\put(-0.93,0.3){$D$}
\put(-0.78,0.1){$\delta$}
\put(-0.35,0.58){$T_1$}
\put(-0.52,0.62){$T_2$}
\caption{\small Basic regions of polynomial type (orange) and transcendental type (blue). Fixed rays and their landing points are shown in grey. In red, the curve $\delta$ and its preimages, which, together with the track boundaries (black) bound the fundamental domains.}
\label{fig:basicregions}
\end{figure}
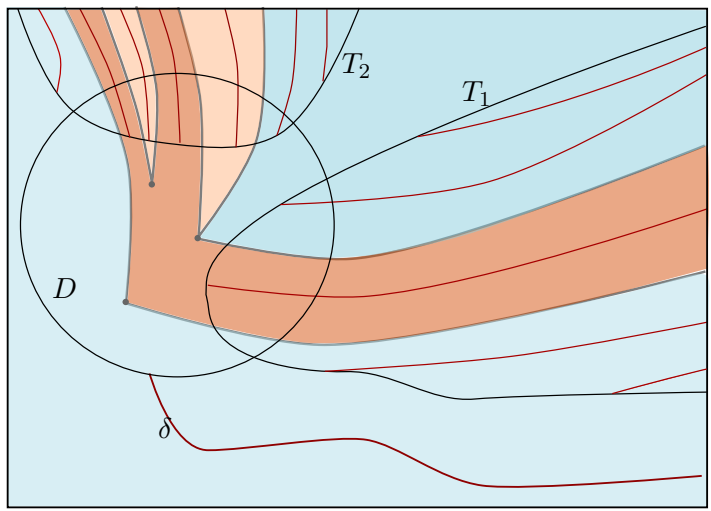

The following proposition is not surprising, and we will not need it  in the sequel, but we include it because it illustrates  the fact that transcendental behaviour, as for example the presence of unbounded Fatou components, only appears associated to basic regions of transcendental type.

\begin{prop}[Bounded Fatou Components]\label{Bounded Fatou Components}
 Let $Q$ be  a  periodic Siegel disk,  a periodic  attracting basin,  or a periodic parabolic basin of period $q$ which is contained in a basic region $B$ of polynomial type for $f^p$,  with $p$ multiple of $q$. Then $Q$ is bounded, and either  $\partial Q\cap \partial B$ is empty or  it is contained in the set of  boundary fixed points. 
\end{prop}

\begin{proof}
Without loss of generality we can assume that the period is 1 and that $B$ is a basic region for $f$.  
Let $z_0$ be the attracting, indifferent or parabolic interior fixed point for $B$. We will only prove the case in which $z_0$ is the center of a Siegel Disk $\De$, since the other two cases are very similar.
 
 Since $B$ is of polynomial type and rays are asymptotically contained in fundamental domains, $B$ intersect finitely many fundamental domains. Indeed, the boundary of $B$ is made of finitely many rays pairs union the point $\infty$. While moving along the boundary of $B$ counterclockwise (passing trough infinity when moving from one ray pair to the other), any two consecutive rays which do not land together belong to the same tract (otherwise $B\setminus \TT$ would be unbounded), hence there are only finitely many fundamental domains between them. Since $\partial B$ contains finitely many ray pairs and fundamental domains have a cyclic order, $B$ intersect finitely many fundamental domains.
 
Let $\gamma$,  $\gamma_\alpha$ be as in  Lemma~\ref{lem:cutting}. Consider the region $\Bhat$ obtained by cutting $B$ with the arcs $\gamma_\alpha$. By choice of $\gamma_\alpha$, $\hat{B}$ is contained in the bounded connected component of $\C\setminus \gamma$, hence  $f(\gamma_\alpha)\cap \hat{B}=\emptyset$. Let us  start by showing that $\De\subset\Bhat$.

Let $\phi:\D\ra\De$ be the Riemann map conjugating $f$ to a rotation, and for all $r<1$ let $V_r:=\phi(\D_r)$. 
Let $r_0:=\sup\{r: V_r\subset\Bhat\}$. Observe that $\De\subset \Bhat$ if and only if $r_0=1$. Suppose that $r_0<1$. Since $\De$ is fully contained in $B$, we would have that $\partial \phi(V_{r_0})\cap \gamma_\alpha\neq\emptyset$. But since $\ov{V_{r_0}}\subset \ov{\Bhat}$ is forward invariant, we have that $f(\ov{V_{r_0}})\subset \ov{\Bhat}$, while   $f(\gamma_\alpha)\subset\C\setminus\ov{\Bhat}$ by choice of $\gamma_\alpha$, this gives a contradiction.
 So $r_0=1$, $\De$ is fully contained in $\Bhat$ and  in particular it is bounded. Since it is bounded its boundary cannot contain escaping points, and since it is contained in $\Bhat$ and forward invariant it cannot intersect $r_\alpha$ (again because $f(r_\alpha)\subset\C\setminus\Bhat$). 
  So $\partial\De\cap \partial B$ contains at most the boundary fixed points.  
\end{proof}
We observe that one could prove Proposition~\ref{Bounded Fatou Components} also by using weakly polynomial-like maps.  
{\begin{rem} It seems to be a difficult problem to know whether there can be (repelling) periodic points on the boundaries of Siegel disks. For polynomials, it is known that if such a case occurs, then the boundary of the Siegel Disk is an indecomposable continuum. By P\'erez-Marco \cite{PM}, there are no periodic points on the boundary of a Siegel disk if the latter has a neighborhood  on which $f$ and $f$ inverse are well defined and univalent (he calls this Siegel Disks of type I).  It is also not unlikely to think that the boundary of unbounded Siegel Disk could contain escaping points.   Observe that rays can not only intersect, but even be contained in the boundaries of attracting basins; for example, when the map $\lambda e^z$ has an attracting fixed point with  a completely invariant basin of attraction, it is known that the Julia set equals the boundary of such basin and consists of curves of escaping points together with their landing points. Proposition~\ref{Bounded Fatou Components} ensures that there cannot be escaping points on the boundary of  Fatou components which are in a basic region of polynomial type.
\end{rem}} 

We now proceed to prove the main result in the paper.

  \begin{proof}[Proof of Theorem~\ref{Singular orbits trapped in basic regions}]
 Let $\mathcal{X}$ be as in the statement. If it is a cycle of attracting or parabolic basins there is nothing to prove since each such cycle  contains a singular orbit. 
 The proof is  a refinement of the classical fact that Cremer points and the boundaries of Siegel disks are contained in the postsingular set (see    Corollary 14.4 in \cite[Corollary 14.4]{Mi} for Siegel Disks and  \cite[Theorem 9.3.4]{Bea} for Cremer points). 
 
{\bf Siegel  case.} 
Let $\mathcal{X}=\Delta_0,\Delta_1,\ldots \Delta_{q-1}$ be a cycle of Siegel disks of period $q$. Fix any $p$ multiple of $q$ and let  $\{B_i\}_{i=0\ldots q}$ be the  $q$ basic regions for $f^p$ containing the images $\Delta_1\ldots \Delta_{q-1}$ of $\Delta=\De_0$, with $f(\Delta_i)=\De_{i+1}$ and all indices are taken modulo $q$. 
Let us define the  set of singular values $S_B:=S(f)\cap (\cup B_i)$, and let us define  
  \[\PP_B:=\ov{\bigcup_{ s\in\SS_B,\ n\leq n_s} f^n(s)},\] 
where $n_s\leq \infty$ is the largest integer such that $f^n_s(s)\in \cup B_i$ for all $n\leq n_s$. 

Let $w\in \dDe_0\setminus \partial B_0$: such a point exists, or otherwise we would have $\De_0=B_0$, which is impossible ($B_0$ contains, for example, rays, which cannot intersect $\Delta_0$).

 Let us first  show that $w\in \PP_B$. Suppose by contradiction that there exists a simply connected neighborhood   $U$ of $w$ with $U\cap\PP_B=\emptyset$. Let us see that this implies that  for each $n$ there is a unique univalent  inverse branch  $\phi_n$ of $f^{-n}$ such that 
$\phi_n(\De_0\cap U)\subset \De_{-n\mod q}$; this would be the same as in \cite{Mi} if we were considering $\PP$ instead of $\PP_B$,  but in our case requires a further argument. Since $w\in \dDe_0\setminus \partial B_0$,  up to restricting $U$ we can assume that  $U\subset B_0$. Observe that any  preimage $V$ of $U$ under  $f^n$ with  $V\cap\De_i\neq\emptyset$ for some $i$, is fully contained in $B_i$. (Indeed, the graph $\Gamma_p$ formed by the closure of the rays invariant under $f^p$ is invariant under $f$.
 So, if  there was a point $z\in \partial( \cup B_i)\cap V\subset( \Gamma_p\cap V)$,   by forward  invariance of  $\Gamma_p$ we would have     $f(z)\in U\cap \Gamma_P$, a contradiction.)
 
 
 So  for each $n$ there is a unique univalent  inverse branch  $\phi_n$ of $f^{-n}$ such that 
$\phi_n(\De\cap U)\subset \De_{-n\mod q}$; the  $\{\phi_n\}$ form a normal family by Montel Theorem because they can be chosen to omit two values (say, a repelling orbit of period at least two which does not intersect $U$). 
 Hence a subsequence $\phi_{n_k}$ converges locally uniformly to an analytic limit map $\phi: U\ra \phi(U)$. 
 
  The map $\phi$ is not constant because no subsequence of $\phi_n$ is  converging to a constant on  $\De_0$. 
On the other hand, let $D\Subset \phi(U)$ be a slightly smaller topological disk still intersecting the Julia set. Then for $n$ large, $\phi_n(U)\supset D$, hence   $f^{n}|_{D}\subset U$ for infinitely many $n$, contradicting the fact that $D$ intersects $J(f)$.

So $\phi_n$ is defined and univalent on $U$ only for $n\leq n_1$ with $n_1$ maximal (possibly, $n_1=0$ if $U$ contains singular values). Since univalent branches in a simply connected open set $V$ are well defined and univalent only if $V$ does not contain singular values, it follows that there is a singular value $s_1$ for $f$ such that $s_1\in \phi_{n_1}(U)$. In particular, $f^n(s_1)\in f^n(\phi_{n_1}(U))\subset \cup B_i$ for all $n\leq n_1$.  
  The same  argument can be repeated for a decreasing sequence $\{U_j\}$ of nested simply connected neighborhoods of $w$, obtaining    an infinite sequence of points $w_j=f^{n_j}(s_j)$ accumulating on $w$ with $s_j\in \SS_B$ and $f^n(s_j)\in \cup B_i$ for $n\leq n_j$.

Now there are two (non-exclusive) cases.     Suppose first that  $s_j=s$ for some $s\in S(f)$ and infinitely many $j$. In this case $n_j\ra\infty$,   $f^n(s)\in \cup B_i$ for all $n$ and  case (1) occurs.  
 If this case does not occur, then there are infinitely many distinct $s_j\in S_B$ such that $f^n(s_j)\in \cup B_i$ for $n\leq n_j$ and we are in case (2). 
 In this case either  $n_j\ra\infty$ or, if $n_j$ has a bounded subsequence,   case (1) occurs.   Indeed, if $n_j$ has a bounded subsequence, there is some minimal $N>0$ and a subsequence $j_k$ such that $f^{N}(s_{j_k})$ converges to $w$. This implies that $s_{j_k}\ra \phi_N(w)$ hence since $\SS(f)$ is closed, $\phi_N(w)\in \SS(f)\cap\partial\Delta$. Since the union of the boundaries of the Siegel disks in the forward orbit of $\Delta$ is forward invariant,  $\phi_N(w)$ satisfies the hypothesis of case (1).

  Since singular values of $f^q$ are  the first $q-1$ images of singular values for $f$, it is directly implied by the construction that each $B_i$ contains a singular value $s_i$ for $f^q$  for which $f^{nq}(s_i)\in B_i$ for all $n\in\N$ such that  $nq\leq n(s_i)$.
  
{\bf Cremer case.} Let $z_0$ be a Cremer fixed point. The proof is the same as in the Siegel case except that the inverse branches $\phi_n$ are defined so as to fix $z_0$, and the  limit function $\phi$   is non-constant because $|\phi_n'(z_0)|=1$ for all $n$.  
 
Now suppose we are in case (1) and let $v\in B_0$ be the singular value for the cycle $\mathcal{X}$. Let $\mathcal{Y}$ be any other interior cycle  of period $\ell$. Let $p=\ell\cdot q$. By the Separation Theorem,  elements of $\mathcal{X}$ and elements of $\mathcal{Y}$ belong to different basic regions for $f^p$ and therefore  the orbit of $v$ is disjoint from the basic regions containing $\mathcal{Y}$.  
 \end{proof}
  
   \begin{rem}\label{Invisible} A repelling periodic point which is not the landing point of any periodic ray of any period is called  \emph{rationally invisible} and is an interior fixed point for $f^p$ for all $p$. It is expected that every repelling periodic orbit is the landing point of at least one but at most finitely many periodic rays of the same period, although partial results are only available when the postsingular set is bounded \cite{Hu}, \cite{BL14}, \cite{BRG17}). Theorem~\ref{Singular orbits trapped in basic regions} together with the Separation Theorem implies that the singular orbits given by case (1)  cannot accumulate on any rationally invisible repelling periodic point, since the latter belongs  to different basic regions for $f^p$ than $\XX$ when $p$ is large enough. \end{rem}

The next Proposition is another practical application of the philosophy that polynomial-type regions are associated to polynomial-type behaviour.

\begin{prop}[Regions of polynomial type] \label{case1forpoltype} Under the assumptions of Theorem~\ref{Singular orbits trapped in basic regions}, 
the first case always occurs if  at least one of the  $B_i$ is a basic region of polynomial-type. In this case  the singular value given by  Theorem~\ref{Singular orbits trapped in basic regions} is in fact a   \emph{critical} value.
\end{prop}
\begin{proof}
  Suppose that  $B_0$ is a basic region of polynomial type. 
Consider the cut region $\Bhat_0\subset B_0$ as in Proposition~\ref{Bounded Fatou Components}. Observe that $\partial\Bhat_0\setminus \partial B_0=\{\gamma_\alpha\}$. In the construction above notice that $\phi_n(U)\subset \Bhat_0$ for all $n$, otherwise $\phi_n(U)$ would contain points in $\gamma_\alpha$ for some minimal $n$ which is impossible because $f(\gamma_\alpha)\cap \Bhat_0=\emptyset$ and $\phi_{n-1}(U)\subset \Bhat_0$.  {This implies that all singular values preventing the continuation of the corresponding inverse branches  are critical values, since finite non-critical preimages of singular values do not obstruct the extension. Since $\Bhat_0$ is bounded it contains  only finitely many critical points,  hence finitely many singular values, and  one of them has to accumulate on $\dDe$ infinitely many times implying the occurrence of  case (1).}
\end{proof}

 From the proof of Theorem~\ref{Singular orbits trapped in basic regions} we obtain the following  corollary, which says that every  basic region $B$ for $f^p$ whose interior fixed point is non-repelling contains either infinitely many singular values for $f^p$, or at least a singular value for $f^p$ which returns to $B$ infinitely many times.
 
\begin{cor}[Singular Values and basic regions]\label{cor:Trapping Theorem} Let $f$ be a polynomial or an entire transcendental map in class $\BBhat$ whose periodic  rays land.    Let $B$ be a   basic region for $f^p$  whose interior fixed point $z_0$ is non-repelling, or which contains an attracting parabolic basin. Then at least one    of the two following   cases occur:
\begin{enumerate}
\item  $B$ contains  at least one singular value $s$ for $f^p$  whose orbit $f^{np}(s)$ is contained in $B$ for all $n\geq0$ and   accumulates either  on the parabolic, attracting  or Cremer fixed point, or on the boundary of the associated Siegel disk.
\item  $B$ contains  infinitely many singular values $s_j$ for $f^p$ such that $f^{pn}(s_j)\in B$ for all $n\leq n_j \to \infty$ as $j\to \infty$, and $\{f^{pn_j}(s_j)\}_{n,j}$ accumulates either  on  the interior periodic point, or on the boundary of the associated Siegel disk. 
\end{enumerate}  

 The first case always occurs if $z_0$ is attracting, or $B$ contains a parabolic basins,  or  $B$ is a basic region of polynomial-type, or $f$ has finitely many singular values. 
\end{cor} 

\begin{proof}[Proof of Corollary~\ref{cor:Trapping Theorem}]Let $B$ be a basic region for $f^p$ whose interior fixed point is non-repelling. If the interior fixed point is attracting,  or if $B$ contains a  parabolic basin invariant under $f^p$, there is nothing to prove, since in this case the cycle of  basins contains a singular orbit for $f$. So we can assume that    $B$ contains either  a Cremer  point $z_0$ or a  Siegel disk $\Delta$ of period $q$, which necessarily divides  $p$ since it is fixed under $f^p$. The claim then follows from   Theorem~\ref{Singular orbits trapped in basic regions}, and the fact that $S(f^p)=S(f)\cup f(S(f))\cup\ldots\cup f^p(S(f))$.  
 
 The fact that polynomial type basic regions are always in case (1) follows from Proposition \ref{case1forpoltype}.
\end{proof}

As the final  corollary of Theorem~\ref{Singular orbits trapped in basic regions} we obtain the improvement on the  classical Fatou Shishikura mentioned in the introduction.

{
\begin{proof}[Proof of Proposition~\ref{FS1}]
Suppose by contradiction that $f$ has $q$ singular values and more than $q$ non-repelling cycles (possibly infinitely many). Take $N+1$ of them and let $p$ be the product of their periods.
Each element in each of the $N+1$ cycles is fixed by $f^p$ hence belongs to a different basic region for $f^p$. In particular,  there are $N+1$ disjoint collections of basic regions which by Theorem~\ref{Singular orbits trapped in basic regions} have to contain each a singular value for $f$ as  well as its entire orbit, giving a contradiction. 

Let  $\mathcal{X}$ be any non-repelling cycle of period $q$ and let $s$ be the singular value given by Theorem~\ref{Singular orbits trapped in basic regions}. Let $\mathcal{Y}$ be any other non-repelling cycle or a  rationally invisible repelling periodic cycle  of period $\ell$. Let $p=\ell\cdot q$. Then elements of $\mathcal{X}$ and elements of $\mathcal{Y}$ belong to different basic regions for $f^p$. By part 1. in Theorem~\ref{Singular orbits trapped in basic regions} the orbit of $s$ is disjoint from the basic regions containing $\mathcal{Y}$.  
\end{proof}
}

\bibliographystyle{amsalpha}
\bibliography{basicregionssv}
\end{document}